\documentclass[reqno,10pt]{amsproc}
\pdfoutput=1

\usepackage{mathptmx}
\usepackage{newcent}
\usepackage{amssymb}

\allowdisplaybreaks[2]

\usepackage{stmaryrd}

\usepackage[all]{xy}

\usepackage{graphicx}

\usepackage{wrapfig}

\newcommand{\blackboardbold}[1]{\ensuremath{\mathbf{#1}}}

\def\Z{\blackboardbold{Z}}
\def\Q{\blackboardbold{Q}}
\def\R{\blackboardbold{R}}
\def\C{\blackboardbold{C}}
\def\H{\blackboardbold{H}}
\def\P{\blackboardbold{P}}
\def\F{\blackboardbold{F}}

\newcommand{\RP}{\ensuremath{\blackboardbold{RP}}}
\newcommand{\CP}{\ensuremath{\blackboardbold{CP}}}

\newcommand{\MU}{\ensuremath{\mathrm{MU}}}

\newcommand{\MO}{\ensuremath{\mathrm{MO}}}
\newcommand{\BO}{\ensuremath{\mathrm{BO}}}

\renewcommand{\H}{\ensuremath{\mathrm{H}}}

\newcommand{\dimp}[1]{\ensuremath{\mathrm{dim}(#1)}}
\newcommand{\codimp}[1]{\ensuremath{\mathrm{codim}(#1)}}

\newcommand{\ol}[1]{\ensuremath{\overline{#1}}}
\newcommand{\til}[1]{\ensuremath{\widetilde{#1}}}

\newcommand{\bd}{\ensuremath{\partial\!}}
\newcommand{\IH}{\ensuremath{\mathrm{IH}}}

\newcommand{\mbf}[1]{\mathbf{#1}}

\newcommand{\iso}{\ensuremath{\cong}}

\newcommand{\tensor}{\otimes}

\newcommand{\into}{\ensuremath{\hookrightarrow}}

\newcommand{\from}{\ensuremath{\leftarrow}}

\newcommand{\codim}{\ensuremath{\mathrm{codim}\;}}

\theoremstyle{plain}
\newtheorem*{theorem*}{Theorem}
\newtheorem{theorem}{Theorem}
\newtheorem*{proposition*}{Proposition}
\newtheorem{proposition}[theorem]{Proposition}
\newtheorem{corollary}[theorem]{Corollary}
\newtheorem*{corollary*}{Corollary}
\newtheorem{lemma}[theorem]{Lemma}
\newtheorem*{lemma*}{Lemma}

\newtheorem*{exercise*}{Exercise}

\newtheorem*{conjecture*}{Conjecture}

\newtheorem*{question*}{Question}

\theoremstyle{definition}

\newtheorem*{definition*}{Definition}

\newtheorem*{example*}{Example}
\newtheorem*{examples*}{Examples}
\newtheorem*{claim*}{Claim}

\newcommand{\sheaf}[1]{\ensuremath{\mbf{#1}}}
\newcommand{\complex}[1]{\ensuremath{#1^\bullet}}
\newcommand{\sheafcomplex}[1]{\complex{\sheaf{#1}}}
\newcommand{\sfcx}{\sheafcomplex}

{

\newcommand{\xto}{\xrightarrow}

\title{Stiefel-Whitney Numbers for Singular Varieties}
\author{Carl McTague}
\email{c.mctague@dpmms.cam.ac.uk}
\address{DPMMS, Wilberforce Road, Cambridge CB3 0WB, England}

\begin{document}

\begin{abstract}
  This paper determines which Stiefel-Whitney numbers can be defined
  for singular varieties compatibly with small resolutions. First an
  upper bound is found by identifying the $\mathbf{F}_2$-vector space
  of Stiefel-Whitney numbers invariant under classical flops,
  equivalently by computing the quotient of the unoriented bordism
  ring by the total spaces of $\mathbf{RP}^3$ bundles. These
  Stiefel-Whitney numbers are then defined for any real projective
  normal Gorenstein variety and shown to be compatible with small
  resolutions whenever they exist. In light of Totaro's result
  \cite{totaro00} equating the complex elliptic genus with complex
  bordism modulo flops, equivalently complex bordism modulo the total
  spaces of $\widetilde{\mathbf{CP}}^3$ bundles, these findings can be
  seen as hinting at a new elliptic genus, one for unoriented
  manifolds.
\end{abstract}

\maketitle

\section*{Introduction}

For a complex algebraic variety $Y$, intersection cohomology provides
groups $\IH^*(Y)$ equipped with an intersection pairing $\IH^*(Y)
\tensor \IH^*(Y) \to \Z$ with the property that if $X \to Y$ is a
small resolution then $\H^*(X) \iso \IH^*(Y)$ as additive groups.
This beautiful fact points to a general philosophy: \emph{Whenever a
  singular variety $Y$ has a small resolution $X \to Y$, the
  invariants of $Y$ should agree with the invariants of $X$.}
According to this philosophy, an invariant can be extended to a
singular variety only if the invariant agrees on all small resolutions
of that variety. Since it is possible to construct a complex algebraic
variety $X$ having two small resolutions $X_1 \to Y \from X_2$ with
$\H^*(X_1) \iso \H^*(X_2)$ as additive groups \emph{but not as rings},
the philosophy for instance says that there is no natural way to
extend the cup product to $\IH^*(Y)$.

For a real algebraic variety $Y$, the situation is as usual more
problematic. First of all, real varieties need not be Witt spaces, so
classical intersection cohomology provides an intersection pairing
between upper and lower middle perversity groups
$\IH^*_{\ol{m}}(Y,\Z/2) \tensor \IH^*_{\ol{n}}(Y,\Z/2) \to \Z/2$,
which generally are not isomorphic. Moreover if $X \to Y$ is a small
resolution then there is not necessarily any relationship between
$\IH^*_{\ol{m}}(Y,\Z/2)$, $\IH^*_{\ol{n}}(Y,\Z/2)$ and $\H^*(X,\Z/2)$.
However, it has recently been shown that if $X_1 \to Y \from X_2$ are
two small resolutions then $\H^*(X_1,\Z/2) \iso \H^*(X_2,\Z/2)$ as
additive groups (compare \cite{totaro02}, \cite{mccrory-parusinski02},
\cite{vanHamel03}). This tantalizing result suggests that there may be
a mod~2 generalization of intersection cohomology for real algebraic
varieties which is compatible with small resolutions whenever they
exist.

This paper applies the above philosophy not to cohomology theories but
rather to characteristic numbers, specifically Stiefel-Whitney
numbers. That is, it investigates which Stiefel-Whitney numbers can be
defined for singular real varieties compatibly with small resolutions.
It begins by analyzing the special case of pairs of small resolutions
related by classical flops. The main result of this paper (stated
without proof in \cite{totaro02}) is that the $\F_2$-vector space of
Stiefel-Whitney numbers invariant under classical flops is spanned by
the numbers $w_1^k w_{n-k}$ for $0 \le k \le n-1$. These numbers are
used to show that the quotient ring of $\MO_*$ by the ideal $I$
generated by differences $X_1-X_2$ of real classical flops is
isomorphic to:
\begin{align*}
  \F_2[\RP^2,\RP^4,\RP^8,\dots]/((\RP^{2^a})^2=(\RP^2)^{2^a} \text{for
    all $a \ge 2$})
\end{align*}

Finally, the numbers $w_1^kw_{n-k}$ are defined for any real
projective normal Gorenstein variety and shown to be compatible with
small resolutions whenever they exist.

These Stiefel-Whitney numbers have arisen before, in Goresky-Pardon's
calculation of the bordism ring of locally orientable $\F_2$-Witt
spaces \cite{goresky-pardon89}. There they appear in the guise
$v_1^{n-2i} v_i^2$ for $1 \le i \le \lfloor n/2 \rfloor$ and are
defined by using local orientability to lift $v_1$ to cohomology and
by using the $\F_2$-Witt condition to lift the Wu class $v_i$ to
intersection cohomology where it can be squared to obtain a homology
class (see \cite{goresky-1984}).

This paper constructs the numbers $w_1^k w_{n-k}$ differently. This
new construction applies to any real projective normal Gorenstein
variety. The algebraic Gorenstein condition corresponds to the
topological local orientability condition \emph{but real projective
  normal Gorenstein varieties need not be $\F_2$-Witt}, as the 3-fold
node discussed below demonstrates (indeed the 3-fold node is
topologically the cone on $S^1 \times S^1$, whereas the cone on an
even dimensional manifold is Witt iff it has no middle-dimensional
homology).

This investigation was inspired by Totaro's investigation
\cite{totaro00} of the analogous question for complex varieties. He
found that the kernel of the complex elliptic genus:
\begin{align*}
  \MU_* \tensor \Q \to \Q[x_1,x_2,x_3,x_4]
\end{align*}
is generated by differences $X_1-X_2$ where $X_1$ and $X_2$ are
related by classical flops. In light of his result, this paper's
findings can be seen as hinting at an elliptic genus for unoriented
manifolds.


\section{Stiefel-Whitney Classes}

\label{section:sw-classes}

Stiefel-Whitney classes measure how twisted a space is. Intuitively,
the total Stiefel-Whitney class of a manifold is the sum (in mod~2
homology) of the cells along which its tangent bundle twists. In
modern terminology they are ``classifying maps seen through the lens
of mod~2 cohomology'': if the tangent bundle of an unoriented manifold
$M^n$ is classified by a map $f : M \to \BO(n)$ (that is, $TM \iso f^*
\gamma_n$ where $\gamma_n$ is the universal $n$-plane bundle over the
classifying space $\BO(n)$, the Grassmann manifold of $n$-planes in
$\R^\infty$), then the Stiefel-Whitney classes $w_i(M)$ of $M$ are the
images under the pullback $f^* : \H^*(\BO(n),\Z/2) \to \H^*(M,\Z/2)$
of the generators of $\H^*(\BO(n),\Z/2) \iso \Z/2[w_1,w_2,\dots,w_n]$
as a $\Z/2$ algebra. That is, $w_i(M)=f^*(w_i)$. This concise
description encapsulates a lot of geometry and history.

We will use a generalization of Stiefel-Whitney classes to singular
spaces inspired by an older and simpler description of Stiefel-Whitney
classes. In his 1935 thesis \cite{stiefel35}, Stiefel defined the
homology class $w_{n-i}(M)$ in $\H_i(M,\Z/2)$ as the singular locus of
a general set of $i+1$ vector fields and conjectured that it could be
defined simply as the sum of all $i$-simplices in the barycentric
subdivision of a triangulation of $M$. Whitney \cite{whitney40} proved
Stiefel's conjecture in 1939 but only published an ``enigmatically
brief and intricate'' sketch of a proof (according to AW Tucker's MR
review).

In seeking a similar combinatorial formula for rational Pontryagin
classes, Cheeger (in collaboration with Simons) rediscovered Stiefel's
proof in 1969. Cheeger's proof \cite{cheeger69} inspired Sullivan to
ask under what conditions the sum of all $i$-simplices in the
barycentric subdivision of a triangulation of a space forms a mod~2
cycle. (To prove Stiefel's conjecture, Cheeger had of course shown
that this is always the case for smooth manifolds.)  Sullivan
(together with Akin) \cite{sullivan69} worked out that this is always
the case if at each point the local Euler characteristic is odd. Such
spaces, which Sullivan called \emph{mod~2 Euler spaces}, could thus be
given Stiefel-Whitney classes even if they were not smooth. Sullivan
thus began to investigate what classes of spaces other than manifolds
were mod~2 Euler spaces. When Sullivan asked Deligne if he could give
an example of a complex algebraic variety not satisfying this
condition, Deligne surprised Sullivan by almost immediately replying
with a convincing argument that no such example existed using
Hironaka's local resolution of singularities. This inspired Sullivan
to work out a ``naive but complicated'' proof \cite{sullivan69} that
all real analytic spaces are mod~2 Euler spaces. Deligne then outlined
a conjectural theory of Chern classes for singular varieties based on
ideas of Grothendieck, which MacPherson worked out in
\cite{macpherson74}. The theory is both elegant and flexible and goes
as follows.

\begin{proposition*}[MacPherson \cite{macpherson74}]
  There is a unique covariant functor $F_\Z$ from compact complex
  algebraic varieties to abelian groups whose value on a variety is
  the group of constructible functions from that variety to the
  integers (a function $V \to \Z$ is \emph{constructible} if it can be
  written as a finite sum $\sum n_i 1_{W_i}$ where each $n_i \in \Z$
  and each $W_i$ is a subvariety of $V$) and whose value $f_*$ on a
  map $f$ satisfies:
  \begin{align*}
    f_*(1_W)(p) = \chi(f^{-1}(p) \cap W)
  \end{align*}
  where $1_W$ is the function that is identically one on the
  subvariety $W$ and zero elsewhere, and where $\chi$ denotes the
  topological Euler characteristic.
\end{proposition*}

\begin{theorem*}[MacPherson \cite{macpherson74}]
  There is a natural transformation from the functor $F_\Z$ to
  integral homology which, on a nonsingular variety $V$, assigns to
  the constant function $1_V$ the Poincar\'e dual of the total Chern
  class of $V$. There is only one such natural transformation.
\end{theorem*}

Explicitly, MacPherson's theorem assigns to any $\Z$-constructible
function $\alpha$ on a compact complex algebraic variety $V$ an
element $c_*(\alpha)$ of $\H_*(V,\Z)$ satisfying the following three
conditions:
\begin{enumerate}
  \item $f_* c_*(\alpha)=c_* f_*(\alpha)$
  \item $c_*(\alpha+\beta) = c_*(\alpha) + c_*(\beta)$
  \item $c_*(1) = c(V) \cap [V]$ if $V$ is smooth
\end{enumerate}
It is the first of these three properties, the pushforward formula
relating $f_* c_*(\alpha)$ to the Euler characteristic of the fibers
of $f$, which makes the theory so useful.

\medskip

There is an analogous theory of Stiefel-Whitney homology classes which
replaces complex varieties with real varieties, integral homology with
mod~2 homology, and $\Z$-constructible functions with
$\Z/2$-constructible functions satisfying the ``local Euler
condition''.  Such functions, called \emph{Euler functions},
generalize mod~2 Euler spaces in the sense that $X$ is a mod~2 Euler
space if and only if $1_X$ is an Euler function. We will rely on an
analytic version of the theory developed by Fu-McCrory
\cite{fu-mccrory97}.  (Fulton-MacPherson developed a PL version within
their bivariant framework \cite{fulton-macpherson81}.)

\begin{proposition*}[Fu-McCrory \cite{fu-mccrory97}]
%
  There is a unique covariant functor $E$ from compact subanalytic
  spaces to abelian groups whose value on a space is the subgroup of
  constructible functions from that space to $\Z/2$ satisfying the
  local Euler condition:
  \begin{align*}
    D(\alpha) = \alpha
  \end{align*}
  where $D$ is the duality operator uniquely defined by the equation:
  \begin{align*}
    D(1_W)(p) = \chi_p(W) = \sum_i (-1)^i \;
    \mathrm{rank}\;\H_i(W,W-p)
  \end{align*}
  and whose value $f_*$ on a map $f$ satisfies:
  \begin{align*}
    f_*(1_W)(p) = \chi(f^{-1}(p) \cap W)
  \end{align*}
  where $1_W$ is the function that is identically one on the
  subanalytic subset $W$ and zero elsewhere, and where $\chi$ denotes
  the topological Euler characteristic.
\end{proposition*}

\begin{theorem*}[Fu-McCrory \cite{fu-mccrory97}]
  There is a natural transformation from the functor $E$ to mod~2
  homology which, on a real analytic manifold $V$, assigns to the
  constant function $1_V$ the Poincar\'e dual of the total
  Stiefel-Whitney class of $V$. There is only one such natural
  transformation.
\end{theorem*}

Explicitly, Fu-McCrory's theorem assigns to any Euler function
$\alpha$ on a compact subanalytic space $V$ an element $w_*(\alpha)$
of $\H_*(V,\Z/2)$ satisfying the following three conditions:
\begin{enumerate}
  \item $f_* w_*(\alpha)=w_* f_*(\alpha)$
  \item $w_*(\alpha+\beta) = w_*(\alpha) + w_*(\beta)$
  \item $w_*(1) = w(V) \cap [V]$ if $V$ is smooth
\end{enumerate}
Again, it is the first of these three properties, the pushforward
formula relating $f_* w_*(\alpha)$ to the Euler characteristics of the
fibers of $f$, which makes the theory so useful.

\section{Stiefel-Whitney Numbers and Unoriented Bordism}

Since Stiefel-Whitney classes of two $n$-folds live in different
cohomology rings, they cannot be compared directly. One way to
compare them is to compare their iterated intersection numbers, that
is the products:
\begin{align*}
  w_I[M] := w_{i_1}(M) \cdots w_{i_r}(M) \in \H^n(M,\Z/2) \iso \Z/2
\end{align*}
where $M$ is an $n$-fold and $I=i_1+\cdots+i_r=n$ is a partition of
$n$. This cohomology class $w_I[M]$ is called the \emph{$I^\text{th}$
  Stiefel-Whitney number} of $M$ since it can be naturally identified
with a number mod~2. The collection of Stiefel-Whitney numbers
$w_I[M]$ as $I$ ranges over all partitions of $n$ can be thought of as
``topological coordinates'' of $M$.

What is the geometric meaning of these coordinates? Two $n$-folds $M$
and $N$ are said to be \emph{bordant} if there is an $(n+1)$-fold $W$
with boundary $\bd W = M \sqcup N$. It is not difficult to prove that
if $M$ and $N$ are bordant then $w_I[M]=w_I[N]$ for all $I$. Thom
proved the much deeper converse: if $w_I[M]=w_I[N]$ for all partitions
$I$ then there exists a manifold $W$ with boundary $\bd W = M \sqcup
N$. Thus Stiefel-Whitney numbers detect equality in what is called the
\emph{unoriented bordism ring} $\MO_*$, the ring consisting of bordism
equivalence classes of unoriented manifolds with addition induced by
disjoint union and multiplication induced by topological product.

In a tremendous feat of creativity and precision, Thom \cite{thom54}
showed that the unoriented bordism ring $\MO_*$ is a polynomial
algebra freely generated over $\Z/2$ by manifolds $Y^n$, one in each
dimension $n$ not of the form $2^j-1$. That is:
\begin{align*}
  \MO_* \iso \Z/2[Y^2,Y^4,Y^5,Y^6,Y^8,\dots]
\end{align*}
The generator $Y^n$ can be taken to be any degree-(1,1) hypersurface
in $\RP^a \times \RP^b$ provided $a+b=n+1$ and the binary expansions
of $a$ and $b$ are disjoint, that is there is no ``carrying'' when
adding them in base~2 (see \cite[Problem~16-F on
p.~197]{milnor-stasheff74}).  Actually, if $n$ is even then $Y^n$ can
be taken simply to be $\RP^n$.

An essential tool in establishing such claims and the claims below is
the characteristic number $s_n(w)[Y^n]$, which equals 1 (mod 2)
precisely when $Y$ is indecomposable in the ring $\MO_*$. It is
defined as follows: If the Stiefel-Whitney classes $w_1,\dots,w_k$ are
viewed as the elementary symmetric polynomials in formal variables
$t_1,\dots,t_N$, then $s_k(w)$ is the power-sum polynomial $t_1^k +
\cdots + t_N^k$ (which, being a symmetric polynomial, can be expressed
as a polynomial in the Stiefel-Whitney classes $w_1,\dots,w_k$). See
\cite[p.~192]{milnor-stasheff74}.

For example, the formula $w_*(\RP^n)=(1+w_1 O(1))^{n+1}$ lets one
think of $w_k(\RP^n)$ as the $k^\text{th}$ elementary symmetric
function in $n+1$ variables, all set to the value $w_1 O(1)$. This
immediately gives the formula $s_k(w)(\RP^n) = (n+1) w_1 O(1)^k$ which
implies that $\RP^\mathrm{even}$ is not bordant to a (nontrivial)
product of manifolds.

\section{Classical Flops}

The simplest example of a variety having two different small
resolutions, discovered by Atiyah \cite{atiyah-1958}, is the 3-fold
node $Y=\{x_1x_2 - x_3x_4=0 \} \subset \P^4$. Near its singular point,
$Y$ is Zariski locally isomorphic to the affine cone of $\sigma (\P^1
\times \P^1) \subset \P^3$ where $\sigma : \P^1 \times \P^1 \into
\P^3$ is the Segre embedding corresponding to the ample line bundle
$O(1,1)=\pi_1^* O(1) \tensor \pi_2^* O(1)$. Blowing up $Y$ at its
singular point therefore gives a smooth resolution $\til{X} \to Y$
with exceptional divisor $\P^1 \times \P^1$ and normal bundle $N_{\P^1
  \times \P^1 / \til{X}} \iso O_{\P^3}(-1)|_{\P^1 \times \P^1} \iso
O(-1,-1)$.

By definition, a map $f : X \to Y$ is \emph{small} if:
\begin{align*}
  \mathrm{codim} \{ y \in Y | \dim f^{-1}(y) \ge r \} > 2r
\end{align*}
for all $r>0$. Since the singular point of $Y$ has 2-dimensional fiber
$\P^1 \times \P^1 \subset \til{X}$ (and since all other points of $Y$
have zero-dimensional fiber), $\til{X} \to Y$ is not small. However,
since the exceptional divisor $\P^1 \times \P^1$ has normal bundle
$O(-1,-1)$, we can contract either $\P^1$ to obtain two small
resolutions $X_1 \to Y \from X_2$ which are projective over $Y$. The
3-folds $X_1$ and $X_2$ are said to be related by an \emph{Atiyah
  flop}.

More generally, if $Y$ is any projective 3-fold which is smooth
everywhere except one point where it is Zariski locally isomorphic to
the 3-fold node then $Y$ has two different small resolutions $X_1 \to
Y \from X_2$. These 3-folds $X_1$ and $X_2$ are said to be related by
a \emph{classical flop}.

Totaro \cite[pp.~770--5]{totaro00} defined \emph{an $n$-dimensional
  classical flop} to be a diagram:
$$\xymatrix@R8pt@C8pt{
    & \til{X} \ar[dl] \ar[dr] \\
    X_1 \ar[dr] && X_2 \ar[dl] \\
    & Y
  }$$
where $Y$ is a singular projective $n$-fold which, near each point of
its singular locus $Z$, is Zariski locally isomorphic to the 3-fold
node times a smooth $(n-3)$-fold. Blowing up $Y$ along $Z$ gives a
smooth resolution $\til{X} \to Y$ whose exceptional divisor is a $\P^1
\times \P^1$ bundle over the smooth $(n-3)$-fold $Z$. Contracting
either family of $\P^1$'s gives two different small resolutions $X_1
\to Y \from X_2$.

Totaro showed that any $n$-dimensional flop $X_1 \to Y \from X_2$ can
be described along $Z$ by rank-2 vector bundles $A,B$ over $Z$, where
the inverse image of $Z$ in $X_1$ is the $\P^1$-bundle $\P(A)$ and has
normal bundle $B \tensor O(-1)$ and where the inverse image of $Z$ in
$X_2$ is the $\P^1$-bundle $\P(B)$ and has normal bundle $A \tensor
O(-1)$. He showed that to any rank-2 algebraic vector bundles $A,B$
over $Z$ there corresponds a classical flop. Moreover, he showed that
in $\MO_*$ the difference $X_1 - X_2$ equals the total space of the
projective bundle $\RP(A \oplus B^*) \to Z$. Thus we can determine how
a Stiefel-Whitney number $w_I$ changes under a classical flop by
computing $w_I[\RP(A \oplus B^*)]$.

Note that the resolutions $X_1 \to Y \from X_2$ of a classical flop
are \emph{crepant}. That is, $Y$ has a line bundle $K_Y$ which pulls
back to the canonical bundles $K_{X_1}$ and $K_{X_2}$. Indeed,
according to Proposition~\ref{prop:gorenstein} below, \emph{any} small
resolution $X \to Y$ is crepant provided $Y$ is projective, normal and
Gorenstein. (The singular space $Y$ of a classical flop is projective
by assumption; it is Gorenstein since near each singular point it is a
hypersurface times a smooth $(n-3)$-fold and is therefore a local
complete intersection; it is normal according to Serre's criterion for
normality since its singular locus $Z$ has codimension~3 (see
\cite[Proposition II.8.23b]{hartshorne77}).)

\section{Unoriented Bordism Modulo Flops}

Now we have enough background to prove the following result, which
Totaro stated without proof in \cite{totaro02}.

\begin{theorem}
  \label{thm:MOmodFlops}
  The $\F_2$-vector space of Stiefel-Whitney numbers which are
  invariant under real flops of $n$-manifolds is spanned by the
  numbers $w_1^k w_{n-k}$ for $0 \le k \le n-1$ modulo those
  Stiefel-Whitney numbers which vanish for all $n$-manifolds. The
  dimension of this space of invariant Stiefel-Whitney numbers, modulo
  those which vanish for all $n$-manifolds, is 0 for $n$ odd and
  $\lfloor n/4 \rfloor+1$ for $n$ even. The quotient ring of $\MO_*$
  by the ideal $I$ of real flops is isomorphic to:
  \begin{align*}
    \F_2[\RP^2,\RP^4,\RP^8,\dots]/((\RP^{2^a})^2=(\RP^2)^{2^a}
    \text{for all $a \ge 2$})
  \end{align*}
\end{theorem}

We prove this in several stages.

\begin{proposition}
  \label{2007-24-aug-w1w-flop-inv}
  The Stiefel-Whitney numbers $w_1^kw_{n-k}$ for $0 \le k \le n-1$ are
  invariant under classical flops, equivalently they vanish on the
  ideal $I$.
\end{proposition}


\begin{proof}[Proof 1]
  We use the theory of Stiefel-Whitney classes discussed in
  Section~\ref{section:sw-classes}. Consider a classical flop $X_1
  \xto{f_1} Y \xleftarrow{f_2} X_2$. Compute:
  \begin{align*}
    w_1^k w_{n-k}[X_i]
    &= \langle w_1(X_i)^k w_{n-k}(X_i), [X_i] \rangle \\
    &= \langle w_1(X_i)^k,  w_{n-k}(X_i) \cap [X_i] \rangle \\
    &= \langle w_1(X_i)^k,  w_{n-k}(1_{X_i}) \rangle \displaybreak[1]\\
    \intertext{Since $f_i : X_i \to Y$ is crepant, this equals:}
    &= \langle f_i^* w_1(Y)^k, w_{n-k}(1_{X_i}) \rangle \\
    &= \langle w_1(Y)^k, {f_i}_* w_{n-k}(1_{X_i}) \rangle \\
    &= \langle w_1(Y)^k, w_{n-k} {f_i}_* (1_{X_i}) \rangle \\
    \intertext{Since $f_i$ restricts to an $\RP^1$-bundle over $Z$ and
      to an isomorphism over $Y-Z$, and since $\chi(\RP^1)=0$, this
      equals:}
    &= \langle w_1(Y)^k, w_{n-k} (1_Y+1_Z) \rangle
  \end{align*}
  Since this is independent of $i$, it follows that $w_1^k
  w_{n-k}[X_1]=w_1^k w_{n-k}[X_2]$.
\end{proof}

\begin{proof}[Proof 2]
  The tangent bundle of $\RP(A \oplus B^*)$ has a splitting:
  \begin{align*}
    0 \to T_\mathrm{rel}^3 \to T\RP(A \oplus B^*) \to \pi^* TZ \to 0
  \end{align*}
  and $T_\mathrm{rel} \oplus O \iso (A \oplus B^*) \tensor O(1) \iso A
  \tensor O(1) \oplus B^* \tensor O(1)$. Let $u=w_1 O(1)$ and let
  $x_1,x_2$ and $x_3,x_4$ denote the Stiefel-Whitney roots of $A$ and
  $B^*$ respectively. Then the Stiefel-Whitney classes of $T\RP(A
  \oplus B^*)$ can be expressed as the elementary symmetric functions
  in the Stiefel-Whitney roots of $TZ$ together with the formal
  variables $x_1+u,\dots,x_4+u$. This gives one too many formal
  variables to be regarded as Stiefel-Whitney roots, but they can
  nevertheless be used to compute:
  \begin{align*}
    w_1^k&w_{n-k}[\RP(A \oplus B^*)] \\&= \int_{\RP(A \oplus B^*)}(
    x_1+u + x_2+u + x_3+u + x_4+u + \pi^* w_1(TZ) )^k w_{n-k}(\RP(A
    \oplus B^*)) \\
    \intertext{The $u$'s cancel modulo 2 to give:}
    &= \int_{Z}\pi_* \big[(x_1 + x_2 + x_3 + x_4 + \pi^* w_1(TZ) )^k
    w_{n-k}(\RP(A \oplus B^*)) \big] \\
    \intertext{Since $\pi_*$ is a map of $\H^*(Z)$-modules and since
      the $x_i$'s pull back from $\H^*(Z)$, this equals:}
    &= \int_{Z} (x_1 + x_2 + x_3 + x_4 + w_1(TZ) )^k \, \pi_* \big[
    w_{n-k}(\RP(A \oplus B^*)) \big]
  \end{align*}
  which equals zero since $\pi_* \big[ w_{n-k}(\RP(A \oplus B^*))
  \big] = 0$. Indeed, the above splitting of $T\RP(A \oplus B^*)$
  implies that:
  \begin{align*}
    \pi_* \big[ w_{n-k}(\RP(A \oplus B^*)) \big] =& \pi_* \big[ \pi^*
    w_{n-k}(TZ) + \pi^* w_{n-k-1}(TZ) \, w_1(T_\mathrm{rel}) \\
    &\quad+ \pi^* w_{n-k-2}(TZ) \, w_2(T_\mathrm{rel}) + \pi^*
    w_{n-k-3}(TZ) \, w_3(T_\mathrm{rel}) \big] \\
    \intertext{and since $\pi_*$ is a map of $\H^*(Z)$-modules which
      decreases degree by 3, this equals:}
    =& w_{n-k-3}(TZ) \, \pi_* \big[ w_3(T_\mathrm{rel}) \big] \in
    \H^{n-k-3}(Z)
  \end{align*}
  which equals zero since $\int_\mathrm{pt} : \H^0(Z) \to \Z/2$ is iso
  and since:
  \begin{align*}
    \int_\mathrm{pt} \pi_* \big[ w_3(T_\mathrm{rel}) \big]
    = \int_{\pi^{-1}(\mathrm{pt}) \iso \RP^3} w_3(T_\mathrm{rel})
    = \int_{\RP^3} w_3(T\RP^3) = \chi(\RP^3) = \text{0 mod 2.}
    \tag*{\qedhere}
  \end{align*}
\end{proof}


\begin{proposition}
  The degree-$2^k$ generators of $\MO_*$ survive to the quotient
  $\MO_*/I$.
\end{proposition}
\begin{proof}
  Any degree-$2^k$ generator of $\MO_*$ (as an $\F_2$-algebra) has
  $s_{2^k}(w) \ne 0$ mod 2. But in characteristic 2:
  \begin{align*}
    s_{2^k}(w) = (t_1^{2^k} + \cdots + t_{2^k}^{2^k}) =
    (t_1+\cdots+t_{2^k})^{2^k} = w_1^{2^k}
  \end{align*}
  which vanishes on $I$ by Proposition~\ref{2007-24-aug-w1w-flop-inv}.
\end{proof}

Now we use flops corresponding to the bundles $A=\pi_1^* O(1) \oplus
\pi_2^* O(1)$ and $B=O^{\oplus 2}$ over $Z=\RP^a \times \RP^b$ for
particular values of $a$ and $b$ to show that all other (suitably
chosen) generators of $\MO_*$ (as an $\F_2$-algebra) do not survive to
the quotient $\MO_*/I$.

\begin{proposition}
  \label{2007-24-aug-kill-gens}
  Consider the projective bundle:
  \begin{align*}
    E = \RP\big(\pi_1^* O(1) \oplus \pi_2^* O(1) \oplus O^{\oplus 2}
    \big)
  \end{align*}
  over $Z=\RP^a \times \RP^b$. Then $s_{a+b+3}[E] \ne 0$ if $a=2^k-2$
  and $b=0,1,\dots,2^k-3$ for any integer $k \ge 2$.

  This gives a sequence of projective bundles:
  \begin{align*}
    E_5,E_6; E_9,E_{10},E_{11},E_{12},E_{13},E_{14};E_{17},E_{18},\dots
  \end{align*}
  with $s_n[E_n]=1$. Since $\MO_*$ has no generators in degrees of the
  form $2^k-1$, this implies that $\MO_*/I$ is generated by the
  manifolds $\RP^{2^k}$, $k \ge 1$.
\end{proposition}
\begin{proof}
  Totaro computed a formula for $s_n(c)(\til{\CP}(A \oplus B))$ in
  \cite[p.~77]{totaro00}:
  \begin{align*}
    \int_Z \quad \sum_{\substack{i_1+i_2+i_3+i_4=n-3\\i_r \ge 0}}
    x_1^{i_1} x_2^{i_2} x_3^{i_3} x_4^{i_4} \left[ (-1)^{i_2}
      \tbinom{n-1}{i_1} + (-1)^{i_1} \tbinom{n-1}{i_2} + (-1)^{i_4+1}
      \tbinom{n-1}{i_3} + (-1)^{i_3+1} \tbinom{n-1}{i_4} \right]
  \end{align*}
  where $x_1,x_2$ are the Chern roots of $A$ and $x_3,x_4$ the Chern
  roots of $B$. Interpreting this mod 2 gives a formula for
  $s_n(w)(\RP(A \oplus B^*))$ in terms of the Stiefel-Whitney roots of
  $A$ and $B$.

  {\newcommand{\ord}{\mathrm{ord}_2}
   \renewcommand{\a}{\alpha_2}

   For $E$: 
   \begin{align*}
     x_1 &= \pi_1^* w_1(O(1)) = g_1 \in \H^1(\RP^{2^k-2} \times \RP^b) \\
     x_2 &= \pi_2^* w_1(O(1)) = g_2 \in \H^1(\RP^{2^k-2} \times \RP^b)
     \\
     x_3 &=x_4=0
   \end{align*}
   so Totaro's formula gives:
  \begin{align*}
    s_{a+b+3}[E] &= \int_Z x_1^{2^k-2} x_2^b \big[
    \tbinom{2^k+b}{2^k-2} + \tbinom{2^k+b}{b} +
    \tbinom{2^k+b}{0} + \tbinom{2^k+b}{0} \big]
    = \tbinom{2^k+b}{2^k-2} + \tbinom{2^k+b}{b} \,.
  \end{align*}
  This number is odd. Indeed the identity $\ord(n!)=n-\a(n)$, where
  $\ord(n!)$ is the number of times $2$ divides $n!$ and $\a(n)$ is
  the number of nonzero digits in the binary expansion of $n$, lets us
  compute:
  \begin{align*}
    \ord \tbinom{2^k+b}{b} &= \ord
    \tfrac{(2^k+b)!}{b! \, 2^k!} \\
    &= \ord((2^k+b)!) - \ord(b!) - \ord(2^k!) \\
    &= 2^k+b-\a(2^k+b)-b+\a(b)-2^k+\a(2^k) \\
    &= \a(b)-\a(2^k+b)+1 \\
    &= 0 \quad \text{since $b<2^k$} \displaybreak[1] \\
    \ord \tbinom{2^k+b}{2^k-2} &= \ord \tfrac{(2^k+b)!}{(2^k-2)! \, (b+2)!} \\
    &= \ord((2^k+b)!) - \ord((b+2)!) - \ord((2^k-2)!) \\
    &= 2^k + b - \a(2^k+b) - \ord((b+2)!) - 2^k + 2 +
    \a(2^k-2) \\
    &= b - \a(2^k+b) - \ord (b+2)! + 2 + (k-1) \\
    &= b - \a(b) - \ord((b+2)!) + k \quad \text{since $b<2^k$} \\
    &= \ord(b!) - \ord((b+2)!) + k \\
    &> 0 \quad \text{since $b \le 2^k-3$.}
  \end{align*}
  This final inequality holds since $(b+2)!/b! = (b+2)(b+1)$ so
  $\ord((b+2)!)-\ord(b!)$ equals $\ord(b+2)$ or $\ord(b+1)$, which are
  both $<k$ since $b \le 2^k-3$. }
\end{proof}

All that remains to prove Theorem~\ref{thm:MOmodFlops} is to determine
what relations the generators $\RP^{2^k}$ of $\MO_*/I$ satisfy. The
key is the following formula.

\begin{lemma}
  \label{2007-24-aug-formula}
  Let $2n=2b_1+\cdots+2b_r$ be a partition. Then:
  \begin{align*}
    w_1^0w_{2n}&[\RP^{2b_1} \times \dots \times \RP^{2b_r}] = 1 \\
    \intertext{and if $2i \ge 2$ has binary expansion $2i=\sum_{k=1}^s
      2^{c_k}$ then:}
    w_1^{2i}w_{2n-2i}&[\RP^{2b_1} \times \dots \times \RP^{2b_r}]
  \end{align*}
  equals the number (mod 2) of sequences $1 \le j_1,\dots,j_s \le r$
  such that the binary expansion of $2b_{j_k}$ contains $2^{c_k}$ for
  all $1 \le k \le s$.
\end{lemma}
\begin{proof}
  The first equality is immediate since:
  \begin{align*}
    w_{2n}\big[\RP^{2b_1} \times \cdots \times \RP^{2b_r}\big]
    &= \prod_{j=1}^r \binom{2b_j+1}{2b_j} = \prod_{j=1}^r
    (2b_j+1) = 1.
  \end{align*}

  The second equality is more subtle:
  \begin{align*}
    w_1^{2i}&w_{2n-2i}[\RP^{2b_1} \times \dots \times \RP^{2b_r}] \\
    &= ((2b_1+1)g_1 + \cdots + (2b_r+1)g_r)^{\sum 2^{c_k}}
    \sigma_{2n-2i}(\overbrace{g_1,\dots,g_1}^{2b_1+1},\cdots,
    \overbrace{g_r,\dots,g_r}^{2b_r+1}) \\
    \intertext{which in characteristic 2 equals:}
    &= \left[ \prod_{k=1}^s (g_1^{2^{c_k}} + \cdots + g_r^{2^{c_k}})
      \right] \sigma_{2n-2i}(g_1,\dots,g_r)\\
    &= \left[ \sum_{1 \le j_1,\dots,j_s \le r} g_{j_1}^{2^{c_1}}
      g_{j_2}^{2^{c_2}} \cdots g_{j_s}^{2^{c_s}} \right]
    \sigma_{2n-2i}(g_1,\dots,g_r)\\
    &= \sum_{1 \le j_1,\dots,j_s \le r} \quad \prod_{j=1}^r
    \binom{2b_j+1}{2b_j-\sum \{ 2^{c_k} \,:\,j_k=j\}} \,.
  \end{align*}
  {\newcommand{\ord}{\mathrm{ord}_2}
    \renewcommand{\a}{\alpha_2}
  Each of these binomial factors satisfies:
  \begin{align*}
    \ord \binom{2b_j+1}{2b_j-\sum_{k\,:\,j_k=j} 2^{c_k}}
    &= \ord \binom{2b_j+1}{2b_j-2^{k_1}-\cdots-2^{k_\ell}} \\
    &= -\a(2b_j+1)+\a(2b_j-2^{k_1}-\cdots-2^{k_\ell}) +
    \a(2^{k_1}+\cdots+2^{k_\ell}+1) \\
    &= -\a(2b_j)+\a(2b_j-2^{k_1}-\cdots-2^{k_\ell}) + \ell
  \end{align*}
  which equals zero iff the binary expansion of $2b_j$ contains
  $2^{k_1}, \dots, 2^{k_\ell}$. Indeed, note that:
  \begin{align*}
    \ord \, n = \ord \frac{n!}{(n-1)!} = \a(n)-\a(n-1)-1
  \end{align*}
  which implies that $\a(n)-\a(n-1) \ge 1$ with equality iff $n$ is
  odd. More generally, $\a(n)-\a(n-2^k) \ge 1$ with equality iff the
  binary expansion of $n$ contains $2^k$. Assuming without loss of
  generality that $k_1<k_2<\cdots<k_\ell$, use this fact repeatedly to
  conclude that:
  \begin{enumerate}
  \item $\a(2b_j)-\a(2b_j-2^{k_\ell}) \ge 1$ with equality
    iff the binary expansion of $b_j$ contains $2^{k_\ell}$.
  \item $\a(2b_j-2^{k_\ell})-\a(2b_j-2^{k_{\ell-1}}-2^{k_\ell}) \ge 1$
    with equality iff the binary expansion of $2b_j-2^{k_\ell}$
    contains $2^{k_{\ell-1}}$, which occurs iff the binary expansion
    of $2b_j$ contains $2^{k_{\ell-1}}$.
  \item[] $\vdots$
  \item[($\ell$)] $\a(2b_j-2^{k_2}-\cdots-2^{k_\ell}) - \a(2b_j-2^{k_1} -
    \cdots-2^{k_\ell})) \ge 1$ with equality iff the binary expansion
    of $2b_j-2^{k_2}-\cdots-2^{k_\ell}$ contains $2^{k_1}$, which
    occurs iff the binary expansion of $2b_j$ contains $2^{k_1}$.
  \end{enumerate}
  Add these together to conclude that
  $\a(2b_j)-\a(2b_j-2^{k_1}-\cdots-2^{k_\ell}) \ge \ell$ with equality
  iff the binary expansion of $2b_j$ contains
  $2^{k_1},\dots,2^{k_\ell}$, as desired.

  \medskip

  Thus the nonzero summands are those corresponding to sequences
  $j_1,\dots,j_s$ such that the binary expansion of $2b_{j_k}$
  contains $2^{c_k}$ for all $1 \le k \le r$, as desired.  }
\end{proof}

A consequence of Lemma~\ref{2007-24-aug-formula} is that:
\vspace{-3ex}
\begin{align*}
  w_1^{2i}w_{2n-2i}[ \RP^{2^a} \times \RP^{2^a} \times \RP^J ] =
  w_1^{2i}w_{2n-2i}[ \overbrace{\RP^2 \times \cdots \times
    \RP^2}^{2^a} \times \RP^J ]
\end{align*}
for all $0 \le 2i \le 2n=2^{a+1} + \sum J$. The final task is to show
that these are precisely the relations in $\MO_*/I$.

First we will show that manifolds of the form:
\begin{align*}
  \RP^2 \times \cdots \times \RP^2 \times \RP^{2^{a_1}} \times \dots
  \times \RP^{2^{a_r}} \quad \text{with $4 \le 2^{a_1} < 2^{a_2} <
    \cdots < 2^{a_r}$}
\end{align*}
are linearly independent in $\MO_*/I$. Note that in a given dimension
$2n$, such manifolds correspond uniquely to integers $0 \le 4j \le
2n$. Namely, given an integer $0 \le 4j \le 2n$ with binary expansion
$4j=2^{a_1}+\cdots+2^{a_r}$, let $J_{2n}(4j)$ denote the partition:
\begin{align*}
  2n=\overbrace{2+\cdots+2}^{=2n-4j}
    +\overbrace{2^{a_1}+\cdots+2^{a_r}}^{=4j}
\end{align*}

An important consequence of Lemma~\ref{2007-24-aug-formula} can then
be stated as follows.

\begin{corollary}
  \label{2007-24-aug-formula-2}
  Let $0 \le 4i, 4j \le 2n$ be integers with binary expansions
  $4i=\sum_l 2^{b_l}$, $4j=\sum_k 2^{a_k}$. Then $w_1^{4i}
  w_{2n-4i}\big[ \RP^{J_{2n}(4j)} \big]$ equals 1 iff $\{b_k\}
  \subseteq \{a_l\}$.
\end{corollary}
\begin{proof}
  This is immediate since if there is a sequence as in
  Lemma~\ref{2007-24-aug-formula} then it is unique.
\end{proof}

\begin{corollary}
  \label{2007-24-aug-lin-ind}
  The manifolds: $$\big\{ \RP^{J_{2n}(4j)} : 0 \le 4j \le 2n \big\}$$
  are linearly independent in $\MO_{2n}/I$.
\end{corollary}
\begin{proof}
  Consider the matrix:
  \begin{align*}
    \Big[ w_1^{4i}w_{2n-4i}\big[ \RP^{J_{2n}(4j)} \big] \Big]_{0 \le
      4i,4j \le 2n}
  \end{align*}
  By changing bases to:
  \begin{align*}
    b_{4j} = \RP^{J_{2n}(4j)} + \sum_{4j'} b_{4j'}
  \end{align*}
  where $4j'$ ranges over all sub-sums of the binary expansion of
  $4j$, we obtain the matrix:
  \begin{align*}
    \Big[ w_1^{4i}w_{2n-4i}\big[ b_{4j} \big] \Big]_{0 \le 4i,4j \le 2n}
  \end{align*}
  which is diagonal by Corollary~\ref{2007-24-aug-formula-2}.

  The result then follows since by
  Proposition~\ref{2007-24-aug-w1w-flop-inv} the Stiefel-Whitney
  numbers $w_1^{4i}w_{2n-4i}$ vanish on $I$ and hence are well-defined
  on the quotient $\MO_{2n}/I$.
\end{proof}

Now we find the relations.

\begin{proposition}
  \label{2007-24-aug-twisted-rel}
  Consider the projective bundle:
  \begin{align*}
    R_{2^{a+1}} = \RP(O(1) \oplus O^{\oplus 3})
  \end{align*}
  over $Z=\RP^{2^{a+1}-3}$. If $2^a \ge 4$ then
  $s_{2^a,2^a}[R_{2^{a+1}}]=1$.
\end{proposition}
\begin{proof}
  The Stiefel-Whitney roots of $O(1) \oplus O^{\oplus 3}$ are all zero
  except $x_1=w_1(O(1))=g \in \H^1(\RP^{2^{a+1}-3})$. This lets us
  compute:
  \begin{align*}
    s_{2^a,2^a}[R_{2^{a+1}}]
    &= \int_Z \pi_* \big[ \left(\tbinom{3}{1} (g^{2^a}+u^{2^a})u^{2^a}
      + \tbinom{3}{2} u^{2^{a+1}}\right) +
      \left(g^{2^a}+u^{2^a}+3u^{2^a}\right) s_{2^a}(TZ) +
      s_{2^a,2^a}(TZ) \big] \\
    &= \int_Z \pi_* \big[ g^{2^a} u^{2^a} + g^{2^a} s_{2^a}(TZ) +
      s_{2^a,2^a}(TZ) \big] \\
    &= \int_Z g^{2^a} \pi_*(u^{2^a}) = \int_{\RP^{2^{a+1}-3}}
    g^{2^{a+1}-3} = 1
  \end{align*}
  since:
  \begin{align*}
    \pi_*(u^{2^a}) = c_{i-(r-1)}(-O(1) \oplus O^{\oplus
      3})=\sum_{\substack{i_1+i_2+i_3+i_4=2^a-3\\ i_j \ge 0}}
    x_1^{i_1} x_2^{i_2} x_3^{i_3} x_4^{i_4} = g^{2^a-3}
  \end{align*}
  and since $2^a \ge 4$.
\end{proof}

The projective bundles $R_{2^a}$ thus give relations in each degree
$2^a \ge 4$ of $\MO_*/I$. These relations are not particularly simple
to state but their existence alone implies the following simpler
relations.

\begin{proposition}
  \label{2007-24-nice-rels}
  The differences $(\RP^{2^a})^2 - (\RP^2)^{2^a}$ for $a \ge 2$ are in
  the ideal~$I$.
\end{proposition}
\begin{proof}
  A generating set for $\MO_*$ as an $\F_2$-algebra is given by the
  manifolds:
  \begin{align*}
    \RP^2, \RP^4, E_5, E_6, \RP^8, E_9, E_{10}, E_{11}, E_{12},
    E_{13}, E_{14}, \RP^{16}, E_{17}, \dots
  \end{align*}
  where the $E_n$'s are the projective bundles provided by
  Proposition~\ref{2007-24-aug-kill-gens}. One can therefore write:
  \begin{align*}
    R_{2^{a+1}} = \sum a_J \RP^J + \text{(terms involving $E_n$'s)}
  \end{align*}
  where $J$ ranges over all partitions of $2^{a+1}$ into powers of
  two.

  Proposition~\ref{2007-24-aug-twisted-rel} established that
  $s_{2^a,2^a}[R_{2^{a+1}}]=1$. This implies that $a_{2^a,2^a}=1$.
  Indeed Thom's formula:
  \begin{align*}
    s_I(w(E \oplus E')) = \sum_{JK=I} s_J(w(E)) \cdot s_K(w(E'))
  \end{align*}
  (see p.~190 of \cite{milnor-stasheff74}) implies that if $m,n\ge1$
  then:
  \begin{align*}
    s_{2^a,2^a}[M^m \times N^n] &= s_{2^a,2^a}[M] + s_{2^a}[M] \cdot
    s_{2^a}[N] + s_{2^a,2^a}[N] \\&=
    \begin{cases}
      1 & \text{if $m=n=2^a$ and $M,N$ are indecomposable in $\MO_*$} \\
      0 & \text{otherwise}
    \end{cases}
  \end{align*}
  Thus $s_{2^a,2^a}[\RP^{2^a,2^a}]=1$, and the only other term on the
  right hand side of the above formula for $R_{2^{a+1}}$ that could
  possibly have $s_{2^a,2^a}$ nonzero is $\RP^{2^{a+1}}$. But direct
  calculation shows that:
  \begin{align*}
    s_{2^a,2^a}[\RP^{2^{a+1}}] &= \tbinom{2^{a+1}+1}{2} = 2^a
    (2^{a+1}+1) = 0
  \end{align*}
  Thus $a_{2^a,2^a}=1$.

  By subtracting the terms involving $E_n$'s from the above formula
  for $R_{2^{a+1}}$, we obtain an element of $I$ of the form:
  \begin{align*}
    \sum a_J \RP^J
  \end{align*}
  with $a_{2^a,2^a}=1$. Suppose the proposition has been proved in
  degrees $8,\dots,2^a$. By subtracting some element of the ideal
  generated by the elements:
  \begin{align*}
    (\RP^{4,4}-\RP^{2,2,2,2}, \dots, \RP^{2^{a-1},2^{a-1}} -
    \RP^{2,\dots,2})
  \end{align*}
  (which by inductive hypothesis lie in $I$) we can obtain an element
  of $I$ of the form:
  \begin{align*}
    \RP^{2^a,2^a} + \sum_J b_J \RP^J
  \end{align*}
  where $J$ now ranges only over partitions of the form
  $J_{2^{a+1}}(4j)$ for $0 \le 4j \le 2^{a+1}$ (including the
  partition $2^{a+1}=2^{a+1}$). Note that in degree~8 no subtraction
  is needed to bring the element into this form.

  Since this element belongs to $I$, its Stiefel-Whitney numbers
  $w_1^{4i}w_{2^{a+1}-4i}$ are zero (by
  Proposition~\ref{2007-24-aug-w1w-flop-inv}). This implies that:
  \begin{align*}
    w_1^{4i}w_{2^{a+1}-4i}[\RP^{2^a,2^a}] = \sum_J b_J \cdot
    w_1^{4i}w_{2^{a+1}-4i}[\RP^J]
  \end{align*}
  for all $0 \le 4i \le 2^{a+1}$.
  Lemma~\ref{2007-24-aug-formula} gives:
  \begin{align*}
    w_1^{4i}w_{2^{a+1}-4i}[\RP^{2^a,2^a}] =
    w_1^{4i}w_{2^{a+1}-4i}[\RP^{2,\dots,2}] =
    \begin{cases}
      1 & \text{if $i=0$} \\
      0 & \text{otherwise}
    \end{cases}
  \end{align*}
  Corollary~\ref{2007-24-aug-lin-ind} showed that the matrix $\Big[
  w_1^{4i}w_{2n-4i}\big[ \RP^{J_{2n}(4j)} \big] \Big]_{0 \le 4i,4j \le
    2n}$ is nonsingular, so $b_{2,\dots,2}=1$ and all other $b_J=0$.
  Thus $\RP^{2^a,2^a}-\RP^{2,\dots,2} \in I$ as desired. The result
  now follows by induction.
\end{proof}

Proposition~\ref{2007-24-nice-rels} implies that any element of
$\MO_{2n}/I$ can be written as a sum of the form $\sum_J a_J \RP^J$
where $J$ ranges over partitions of the form $J_{2n}(4j)$ for $0 \le
4j \le 2n$. Corollary~\ref{2007-24-aug-lin-ind} showed that these
spaces $\RP^J$ are linearly independent in $\MO_*/I$. Thus the ring
$\MO_*/I$ is completely described. Note that there are precisely
$\lfloor 2n/4 \rfloor + 1$ such spaces $\RP^J$ since they correspond
to integers $0 \le 4i \le 2n$.

\section{Real Projective Varieties with Gorenstein Singularities}

Above we determined that the $\F_2$-vector space of Stiefel-Whitney
numbers invariant under classical flops is spanned by the numbers
$w_1^k w_{n-k}$ for $0 \le k \le n-1$. Now we define these numbers for
any real projective normal Gorenstein variety and show that this
definition is compatible with small resolutions whenever they exist.

Throughout this section let $Y$ be an $n$-dimensional projective
normal Gorenstein variety defined over $\R$. Let $Y(\R)$ and $Y(\C)$
denote the set of real and complex points of $Y$ equipped with the
classical topology. Let $i_Y$ denote the inclusion of $Y(\R)$ into
$Y(\C)$ as the fixed-point set of the involution induced by complex
conjugation. For any morphism $f : X \to Y$, let $f_\R : X(\R) \to
Y(\R)$ and $f_\C : X(\C) \to Y(\C)$ denote the corresponding maps.

We will define the numbers $w_1^k w_{n-k}[Y]$ in two stages. First we
will construct a cohomology class $w_1(Y) \in \H^1(Y(\R),\Z/2)$ and
show that $w_1(X(\R))=f_\R^* w_1(Y)$ for any small resolution $f:X \to
Y$. Second we will construct homology classes $w_{n-k}(Y) \in
\H_k(Y(\R),\Z/2)$ for $0 \le k \le n-1$ and show that
$w_{n-k}(Y)=f_{\R*}\big( w_{n-k}(X(\R)) \cap [X(\R)]\big)$ for any
small resolution $f:X \to Y$. Then we can define:
\begin{align*}
  w_1^k w_{n-k}[Y] :=\langle w_1(Y)^k, w_{n-k}(Y) \rangle
\end{align*}
and it will follow that if $f:X \to Y$ is a small resolution then:
\begin{align*}
  w_1^kw_{n-k}[Y]
  &=\langle w_1(Y)^k, w_{n-k}(Y) \rangle \\
  &= \langle w_1(Y)^k, f_{\R*} (w_{n-k}(X(\R)) \cap [X(\R)]) \rangle \\
  &= \langle f_\R^* w_1(Y)^k, w_{n-k}(X(\R)) \cap [X(\R)] \rangle \\
  &= \langle w_1(X(\R))^k, w_{n-k}(X(\R)) \cap [X(\R)] \rangle \\
  &= \langle w_1(X(\R))^k w_{n-k}(X(\R)), [X(\R)] \rangle \\
  &= w_1^k w_{n-k}[X(\R)] \, .
\end{align*}

Note that, as our notation suggests, the classes $w_1(Y)$ and
$w_{n-k}(Y)$ will not be determined solely by the topology of the real
points $Y(\R)$. They will depend on the algebraic structure of $Y$, in
particular how $Y(\R)$ sits within $Y(\C)$.

\smallskip 

\subsection*{The cohomology class \emph{w}$_\text{1}$(\emph{Y})}

The Gorenstein assumption is the key to defining $w_1(Y)=w_1(K_Y) \in
\H^1(Y(\R),\Z/2)$ where $K_Y$ is the line bundle provided by the
following proposition.

\begin{proposition}
  \label{prop:gorenstein}
  If an $n$-dimensional real projective variety $Y$ is normal and
  Gorenstein then any small resolution $f:X \to Y$ is crepant. That
  is, the canonical bundle over the smooth locus of $Y$ extends to a
  line bundle $K_Y$ over $Y$ and $f^* K_Y \iso K_X$ for any small
  resolution $f : X \to Y$.
\end{proposition}

\begin{proof}
  Since $Y$ is Gorenstein, its dualizing sheaf $\omega_Y$ is a line
  bundle. Since $Y$ is normal and projective, $\omega_Y$ is isomorphic
  to the canonical sheaf $O_Y(K_Y) = O_Y(i_* K_{Y_{sm}})$ where
  $i:Y_{sm} \into Y$ is the inclusion of the smooth locus (see
  \cite[Proposition~5.75]{kollar-mori98}). Thus the canonical sheaf
  $O_Y(K_Y)$ is in fact a line bundle; denote it by $K_Y$. It
  restricts to $K_{Y_{sm}}$.

  Now let $f:X \to Y$ be a small resolution and let:
  \begin{align*}
    Y_r = \{ y \in Y : \dim f^{-1}(y) \ge r \} \, .
  \end{align*}
  Since $f$ is small, $\codimp{Y_r}>2r$ for all $r>0$. In particular
  $f$ has 0-dimensional fibers away from the subspace $Y_1$. Since
  $Y$ is normal, Zariski's Main Theorem says that the fibers of $f$
  are connected, so $f$ is in fact an isomorphism away from $Y_1$.
  That is, there is a commutative diagram:
  \begin{align*}
    \xymatrix{
      X-f^{-1}(Y_1) \ar@{^{(}->}[r] \ar[d]_\iso
      & X \ar[d]^f \\
      Y-Y_1 \ar@{^{(}->}[r] & Y }
  \end{align*}

  All that remains is to show that $f^*K_Y \iso K_X$. Since $Y-Y_1
  \subset Y_{sm}$, the above diagram implies that $f^* K_Y$ and $K_X$
  both restrict to the canonical bundle over $X-f^{-1}(Y_1)$. Since
  $X$ is smooth, $X$ is normal. It is a standard result that two line
  bundles on a normal variety over a field are isomorphic if they are
  isomorphic away from a subset of codimension $\ge2$. Thus to show
  that $f^* K_Y \iso K_X$ it suffices to show that $\codim
  f^{-1}(Y_1)\ge2$.
  Write:
  \begin{align*}
    f^{-1}(Y_1) = f^{-1}(Y_1-Y_2) \cup f^{-1}(Y_2-Y_3) \cup \cdots
    \cup f^{-1}(Y_R-Y_{R+1})
  \end{align*}
  Since $\dimp{X}=\dimp{Y}$ and $f$ is small:
  \begin{align*}
    \codim f^{-1}(Y_r-Y_{r+1}) \ge \codimp{Y_r-Y_{r+1}}-r > 2r-r=r
  \end{align*}
  for any $r>0$. Thus $\codim f^{-1}(Y_1) = \text{min}_{r>0} \; \{
  \codim f^{-1} (Y_r-Y_{r+1}) \} \ge 2$ as desired.
\end{proof}


\subsection*{The homology classes \emph{w$_\text{n--k}$}(\emph{Y})}

To define the homology classes $w_{n-k}(Y)$ for $0 \le k \le n-1$, we
will define an Euler function $\alpha_Y$ and set:
\begin{align*}
  w_{n-k}(Y) := w_{n-k}(\alpha_Y) \in \H_k(Y(\R),\Z/2)
\end{align*}
where $w_* : F_{\Z/2} \to \H_*(-,\Z/2)$ is the natural transformation
described in Section~\ref{section:sw-classes}.

Namely let $\alpha_Y$ be the $\Z/2$-valued function which assigns to
each point of $Y(\R)$ its local mod~2 intersection homology Euler
characteristic \emph{within the complexification $Y(\C)$}. More
formally let:
\begin{align*}
  \alpha_Y = \chi(\mbf{IC}^\bullet_{Y(\C),\Z/2}) \circ i_Y
\end{align*}
where $\mbf{IC}^\bullet_{Y(\C),\Z/2}$ denotes the, say upper-middle,
mod~2 intersection chain sheaf on $Y(\C)$ (see \cite{ih2}), and $\chi$
denotes the function which assigns to a sheaf $\sfcx{A}$ the
$\Z/2$-valued function:
\begin{align*}
  p \mapsto \sum_i \mathrm{rank}\; \mathbf{H}^i(\sfcx{A})_p \quad \text{mod 2}
\end{align*}

To see that this $\Z/2$-valued function $\alpha_Y$ is constructible
and satisfies the local Euler condition, consider the complex
algebraic variety $Y(\C)$. It has a complex algebraic stratification
(see for instance \cite[p.~43]{smt88}) and the real points of its
strata form real algebraic subvarieties $\{W_i\}$ of $Y(\R)$ (these
may not form a stratification of $Y(\R)$ but that is
irrelevant). Since the intersection chain sheaf
$\mbf{IC}^\bullet_{Y(\C),\Z/2}$ is constructible with respect to the
stratification of $Y(\C)$, the function $\alpha_Y$ is a $\Z/2$-linear
combination of the characteristic functions $\{1_{W_i}\}$. As
explained in Section~\ref{section:sw-classes}, the fact that real
analytic spaces are mod~2 Euler spaces implies that these
characteristic functions, and hence $\alpha_Y$ satisfy the local Euler
condition, as desired.

\begin{proposition}
  If $f : X \to Y$ is a small resolution then:
  \begin{align*}
    f_{\R*}(w_*(X(\R)) \cap [X(\R)]) = w_*(\alpha_Y) \in
    \H_*(Y(\R),\Z/2)
  \end{align*}
\end{proposition}
\begin{proof}
  By definition:
  \begin{align*}
    \alpha_Y = \chi( \sfcx{H}(\sfcx{IC}_{Y(\C),\Z/2}) ) \circ i_Y &=
    \chi \Big( \lim_{U \ni p} \H^\bullet(U, \sfcx{IC}_{Y(\C),\Z/2})
    \Big) \circ i_Y \\
    &\stackrel{(1)}{=} \chi \Big( \lim_{U \ni p} \H^\bullet(U, Rf_{\C*}
    (\Z/2)_{X(\C)})
    \Big) \circ i_Y \\
    &= \chi \Big( \lim_{U \ni p} \H^\bullet(f_\C^{-1}(U),
    (\Z/2)_{X(\C)})
    \Big) \circ i_Y \\
    &= \chi \Big( \H^\bullet(f_\C^{-1}(p), \Z/2) \Big) \circ i_Y \\
    &\stackrel{(2)}{=} \chi \Big( \H^\bullet(f_\R^{-1}(p), \Z/2) \Big) \circ i_Y \\
    &= f_{\R*}(1_{X})
  \end{align*}

  Equality~(1) holds since Goresky-MacPherson showed that $Rf_{\C*}
  (\Z/2)_{X(\C)} \iso \sfcx{IC}_{Y(\C),\Z/2}$ (see
  \cite[p.~121]{ih2}). Equality~(2) holds since the Borel-Moore
  homology long exact sequence (see \cite[p.~371]{fulton98}):
  \begin{align*}
    \cdots \to \H_j \big( f_\R^{-1}(p) \big) \to \H_j \big( f_\C^{-1}(p)
    \big) \to \H_j \big( f_\C^{-1}(p)-f_\R^{-1}(p) \big) \to \cdots
  \end{align*}
  implies that:
  \begin{align*}
    \chi(f^{-1}_\C(p))-\chi(f^{-1}_\R(p))=\chi\big(f^{-1}_\C(p)-f^{-1}_\R(p)\big)
  \end{align*}
  which is even since complex conjugation induces a free involution of
  $f_\C^{-1}(p) - f_\R^{-1}(p)$ (and the Euler characteristic is
  multiplicative for fiber bundles, in particular for double covers).

  The result then follows from the pushforward formula $f_{\R*} w_* =
  w_* f_{\R*}$.
\end{proof}

\smallskip

The definition of $\alpha_Y$ might seem unnecessarily complicated. Why
not simply use $1_Y$ instead? The 3-fold node $Y$ illustrates why
not. By the preceding proposition:
\begin{align*}
  w_3[X_1]=w_3[X_2] &= w_3(\alpha_Y)
  \intertext{where $X_1 \to Y \from X_2$ are its small resolutions
    discussed earlier. These resolutions are isomorphisms away from
    the singular point $i:P\to Y$, where the fiber $\P^1$ has mod 2
    Euler characteristic $\chi(\P^1(\R))=0=2=\chi(\P^1(\C))$. Thus:}
  &= w_3(1_Y-1_P) = w_3(1_Y) - w_3(1_P) 
  \intertext{Since $w_3(1_P)$ lives in $\H_0(Y(\R),\Z/2)$ and $w_*(1_P)
    \cap [P] = 1 \in \H_*(P(\R),\Z/2)$, we can use the pushforward
    formula to compute:}
  &= w_3(1_Y) - w_3(i_* 1_P) = w_3(1_Y) - i_* w_0(1_P) \\
  & = w_3(1_Y) - 1
\end{align*}
Thus if one used $1_Y$ instead of $\alpha_Y$ then the resulting
definition of $w_3[Y]$ would not equal $w_3[X_i]$ and would thus be
incompatible with small resolutions.

\section*{Acknowledgments}

Thanks to Burt Totaro for inspiration and many helpful
conversations. Thanks to the referee for a careful reading and several
helpful comments. Thanks also to the National Science Foundation, the
Cambridge Overseas Trusts, the Cambridge Philosophical Society and the
Cambridge Lundgren Fund for their generous financial support.


\end{document}